\documentclass[12pt]{article}
\usepackage{pslatex}
\usepackage{amsmath,amsthm,amssymb,amsfonts, stmaryrd}
\usepackage{fancybox}
\usepackage{color, graphicx}
\usepackage{url}
\newtheorem{theorem}{Theorem}[section]

\newtheorem{conjecture}[theorem]{Conjecture}
\newtheorem{corollary}[theorem]{Corollary}

\newtheorem{definition}[theorem]{Definition}
\newtheorem{example}[theorem]{Example}

\newtheorem{lemma}[theorem]{Lemma}
\newtheorem{notation}[theorem]{Notation}

\newtheorem{proposition}[theorem]{Proposition}
\newtheorem{remark}[theorem]{Remark}

\newcommand{\Z}{\mathbb Z}

\newcommand{\N}{\mathbb N}
\newcommand{\Q}{\mathbb Q}
\newcommand{\R}{\mathbb R}
\newcommand{\K}{\mathbb K}

\newcommand{\vs}[1]{\langle #1 \rangle}

\begin{document}

\title{Wilf's conjecture and Macaulay's theorem}
\date{August 2015}

\author{S.~Eliahou}

\maketitle

\begin{abstract}
Let $S \subseteq \N$ be a numerical semigroup with multiplicity $m=\min(S\setminus \{0\})$, conductor $c=\max(\N \setminus S)+1$ and minimally generated by $e$ elements. Let $L$ be the set of elements of $S$ which are smaller than $c$. Wilf conjectured in 1978 that $|L|$ is bounded below by $c/e$. We show here that if $c \le 3m$, then $S$ satisfies Wilf's conjecture. Combined with a recent result of Zhai,  this implies that the conjecture is asymptotically true as the genus $g(S)=|\N \setminus S|$ goes to infinity. One main tool in this paper is a classical theorem of Macaulay on the growth of Hilbert functions of standard graded algebras.
\end{abstract}



\section{Introduction}
A \textit{numerical semigroup} is a subset $S \subseteq \N$ closed under addition, containing 0 and of finite complement in $\N$. The elements of $\N \setminus S$ are called the \textit{gaps} of $S$. The largest gap is denoted $F(S) = \max (\N \setminus S)$ and is called the \emph{Frobenius number} of $S$. The integer $c(S)=F(S)+1$ is known as the \textit{conductor} of $S$. It satisfies $c(S)+\N \subseteq S$ and is minimal for that property. The number of gaps $g(S)=|\N \setminus S|$ is known as the \textit{genus} of $S$, and the smallest nonzero element $m(S) = \min (S\setminus \{0\})$ as the \textit{multiplicity} of $S$. 

Every numerical semigroup $S$ is finitely generated, i.e. is of the form 
$$
S \ = \ \vs{a_1,\dots,a_n} \ = \ \N a_1+ \dots + \N a_n
$$
for suitable globally coprime integers $a_1,\dots,a_n$. The least number $n$ of generators of $S$ is denoted $e=e(S)$ and is called the \textit{embedding dimension} of $S$.

Is there a general upper bound for the density of the gaps of $S$ in the integer interval $[0,c(S)-1]$? This question was asked by Wilf in \cite{W} where, more precisely, he asked whether for $S=\vs{a_1,\dots,a_n}$ the bound 
$$
\frac{|\N \setminus S|}{c(S)} \ \le \ 1-1/n
$$
might always hold\footnote{Of course, the question is sharpest when $n=e(S)$, the embedding dimension of $S$.}. This question is still widely open and is often referred to as Wilf's conjecture, in the following equivalent form. We shall denote $L(S)=S \cap [0,c(S)-1]$ thoughout, where `L' stands for \emph{left part} relative to the conductor.
\begin{conjecture}[Wilf] Let $S$ be a numerical semigroup generated by $n$ elements. Then
$$
\frac{|L(S)|}{c(S)} \ \ge \ \frac{1}{n}.
$$
\end{conjecture}
The equivalence between the two formulations plainly follows from the formulas 
$$
|L(S)|+|\N \setminus S| \ = \ \large | \, [0,c-1] \, \large | \ = \ c,
$$
where $c=c(S)$. Wilf gave the following example where equality holds in his conjecture:
$$
S \ = \ \{0\} \cup (m+\N) \ = \ \{0,m,m+1, \dots\}
$$
for some integer $m \ge 2$. Indeed in this case, one has $|L(S)|=1$, $c(S)=m$, and $e(S)=m$ since $S$ is minimally generated by $\{m,m+1,\dots, 2m-1\}$.

Another equality case in Wilf's conjecture is when $e(S)=2$, i.e. for two-generated numerical semigroups $S=\vs{a,b}$ with $\gcd(a,b)=1$. Indeed, nearly a century before the formulation of the conjecture, Sylvester showed in \cite{Sy} that one has $c(S)=(a-1)(b-1)$ and $|L(S)| = c(S)/2$ in this case.

Finally, the last known equality case in Wilf's conjecture is the following:
$$
S \ = \ m\N \cup (qm+\N) \ = \ \{0,m,2m,\dots,(q-1)m, qm, qm+1, qm+2,\dots\}
$$
for given integers $m,q \ge 1$. Indeed in this case, one has $|L(S)|=q$, $c(S)=qm$, and $e(S)=m$ since $S$ is minimally generated by $\{m,qm+1,qm+2, \dots, qm+m-1\}$. This case actually generalizes the first one by taking $q=1$. 

It is not known whether these are the only equality cases in Wilf's conjecture, but all independent computer experiments so far suggest that the above list might well be complete. See e.g. Question 8 in \cite{MS}.

\medskip

Wilf's conjecture has been shown to hold under various hypotheses, including in \cite{Sy} for $e=2$ as mentioned above, in \cite{FGH} for $e=3$, in \cite{DM} for $|L| \le 4$, by computer in \cite{Br08} for genus $g \le 50$ and more recently in \cite{FH} for $g \le 60$, in \cite{K} for $c \le 2m$, and in \cite{Sa} for $e \ge m/2$ and for $m \le 8$.

\medskip
In this paper, we extend the verification of Wilf's conjecture to all numerical semigroups $S$ satisfying $c \le 3m$, and in some other circumstances. The importance of the former case stems from a recent result of Zhai stating that, asymptotically as the genus $g(S)$ goes to infinity, the proportion of numerical semigroups $S$ satisfying $c(S) \le 3m(S)$ tends to 1 \cite{Z}. In a forthcoming paper, we will show that Wilf's conjecture holds for all numerical semigroups $S$ satisfying $|L(S)| \le 10$.

\smallskip

One key tool in the present paper is a suitable version of Macaulay's classical theorem on the growth of Hilbert functions of standard graded algebras.

\smallskip
Here are a few more details on the contents of this paper. Section~\ref{notation} is devoted to basic notation and notions used throughout the paper. In Section~\ref{partition}, we study a convenient partition of a numerical semigroup $S$ by its intersections with translates of the integer interval $[c,c+m-1]$, and we introduce the \textit{profile} of $S$. A brief Section~\ref{apery} gives some useful formulas in terms of Ap\'ery elements with respect to $m$. Section~\ref{hilbert} recalls some background material on standard graded algebras, Hilbert functions and Macaulay's theorem, and proposes a condensed version thereof which is well-suited to our subsequent applications to Wilf's conjecture. Section~\ref{wilf} is the heart of the paper, where all the material developed in the preceding sections is used to settle Wilf's conjecture in the case $2m < c \le 3m$. A few more cases of the conjecture are then settled in the last Section~\ref{further}.

\smallskip
Nice books are available for background information on numerical semigroups. See \cite{R, RG}.

\section{More notation}\label{notation}

In this paper we shall mostly use \textit{integer intervals}, not real ones, except in Section~\ref{hilbert}. So, for rational numbers $x,y \in \Q$, we shall denote
\begin{eqnarray*}
[x,y] & = & \{n \in \Z \mid x \le n \le y\},
\end{eqnarray*}
\vspace{-0.9cm}
\begin{eqnarray*}
[x,y[ & = & \{n \in \Z \mid x \le n < y\}.
\end{eqnarray*}
In particular, if $y \in \Z$ then $[x,y[\; = [x,y-1]$ and $\big| [x,y[ \big|=y-x$. We shall also denote $[x,\infty[ \; = \{n \in \Z \mid n \ge x\}$.

\subsection{Primitives and decomposables}
Let $S$ be a numerical semigroup. We shall denote $S^* \ = \ S \setminus \{0\}$.
 
\begin{definition} We say that the element $x \in S^*$ is \emph{decomposable} if 
$$
x \ = \ x_1+ x_2
$$
for some  $x_1,x_2 \in S^*$, \emph{primitive} otherwise\footnote{Other commonly used terms for \emph{primitive element} are \emph{irreducible element} or \emph{atom}.}. We denote by $D = D(S)$ the set of decomposable elements in $S^*$, and by $P = P(S)$ its set of primitive elements. Thus $S^* = P \mathbin{\dot{\cup}} D$, the disjoint union of $P$ and $D$.
\end{definition}

Denoting $A+B \ = \ \{a+b \mid a \in A, b \in B\}$ the \textit{sum} of two subsets $A,B \subseteq \Z$, or simply $a+B$ if $A=\{a\}$, we have 
$$
D \ = \ S^*+S^*, \quad P \ = \ S^* \setminus D. 
$$
Clearly, every element $x \in S^*$ may be expressed as a finite sum of primitive elements. That is, the set $P$ generates $S$ as a semigroup. In fact, $P$ is the unique \textit{minimal generating set} of $S$, since every generating set of $S$ necessarily contains $P$. 

The finiteness of $P$, i.e. of the embedding dimension $e=|P|$, follows from the inclusion $P \subseteq [m,c+m[$, which itself is due to the inclusions
$$
[c+m,\infty[ \; \ = \ m+[c,\infty[ \; \ \subseteq \ m + S^* \ \subseteq \ S^* + S^* \ = \ D.
$$
Alternatively, one has $|P| \le m$, since any two distinct primitive elements of $S$ cannot be congruent mod $m$.

%
%
%

\subsection{The associated constants $q$, $\rho$ and $W(S)$}

The following constants associated to $S$ will be used throughout the paper, often tacitly so.
\begin{notation}\label{q,rho} Let $S$ be a numerical semigroup. We denote by $q=q(S)$ and $\rho=\rho(S)$ the unique integers satisfying
$$
c \ = \ qm-\rho
$$
with remainder $\rho \in [0,m[$. That is, we set \ $q = \left\lceil c/m \right\rceil$ and $ \rho = qm-c.$
\end{notation}

\begin{example} If $q=1$, then $\rho = 0$, and $c=m$ since $c \ge m$ always.
The semigroup structure of $S$ is very simple in this case, namely
$$S \ = \ \{0\} \cup [c, \infty[.$$
\end{example}
This case was met above already, as the first example of equality in Wilf's conjecture.

\begin{example} If $q=2$, then $m < c \le 2m$. As mentioned above, Wilf's conjecture holds in this case as well \cite{K}. See below for a new simpler proof.
\end{example}

Thus, Wilf's conjecture holds for $q \le 2$. In this paper, we extend this result to the much more demanding case $q=3$.

\begin{notation} Let $S$ be a numerical semigroup. We denote
$$
W(S) \ = \ e(S)|L(S)|-c(S).
$$
\end{notation}
It allows us to reformulate Wilf's conjecture in the following equivalent way.
\begin{conjecture} Let $S$ be a numerical semigroup. Then
$
W(S) \ge 0.
$
\end{conjecture} 
The new results presented in this paper have been obtained via this formulation, by a successful evaluation of $W(S)$ in the cases under consideration.

\section{A convenient partition} \label{partition}
Throughout this section, $S$ denotes a numerical semigroup with multiplicity $m$, conductor $c$ and associated constants $q, \rho$.

\subsection{The interval $[c,c+m[$}
The integer interval $[c,c+m[$ of cardinality $m$ is entirely contained in $S$ and plays a special role in our present approach. We shall denote it by
$$
I_{q} \ = \ [c, c+m[.
$$
More generally, we shall consider the various translates of $I_q$ by multiples of $m$. 
\begin{notation} For $j \in \Z$, we denote by $I_j$ the \emph{translate} of $I_q$ by $(j-q)m$, i.e.
\begin{eqnarray*}
I_j & = & I_q + (j-q)m \\
& = & [c-(q-j)m, c-(q-j-1)m[ \\
& = & [jm-\rho, (j+1)m-\rho[.
\end{eqnarray*}
\end{notation}
For instance, we have
$$
I_{q-1} \ = \ [c-m,c[, \quad 
I_{1} \ = \ [m-\rho, 2m-\rho[, \quad 
I_{0} \ = \ [-\rho,m-\rho[.
$$
As the various $I_j$ for $j \ge q+1$ need not be distinguished here, we denote
$$
I_{\infty} \ = \ \bigcup_{j \ge q+1} I_j  \ = \ [c+m, \infty[.
$$
The partition of $S$ induced by the intervals $I_j$'s will be used throughout.
\begin{notation} For all $j \ge 0$, we denote
$$S_j \ = \ S \cap I_j \ = \ S\cap [jm-\rho, (j+1)m-\rho[.$$
\end{notation}
Note the following straightforward properties:
\begin{eqnarray*}
jm & \in & S_j  \quad \forall j \ge 0, \\
S_0 & = & S \cap [-\rho, m-\rho[ \ \ =\ \{0\}, \\
S_1 & \subseteq &  [m, 2m-\rho[,  \quad \textrm{(as $\min S_1 = m$)} \\
S_{q-1} & \subsetneq & I_{q-1}, \quad \textrm{(as $c-1 \in I_{q-1} \setminus S$)} \\
S_{q+j} & = & I_{q+j} \quad \forall j \ge 0.
\end{eqnarray*}
\begin{lemma}\label{formula for |L|}
Let $L=L(S)=S \cap [0,c[$. We have
\begin{eqnarray*}
L & = & S_0 \, \mathbin{\dot{\cup}} \, S_1  \, \mathbin{\dot{\cup}} \,  \cdots  \, \mathbin{\dot{\cup}} \, S_{q-1}, \\
|L| & = & 1 + |S_1| + \dots + |S_{q-1}|.
\end{eqnarray*}
\end{lemma}
\begin{proof} Straightforward from the definitions, since $L \subseteq [0,c[ \, \subseteq  \mathbin{\dot{\bigcup}_{0 \le j \le q-1}} I_j$.
\end{proof}
\begin{lemma} We have
$$
m + S_j \ \subseteq \ S_{j+1} \quad \textrm{ for all } \: j \ge 0
$$
and, in particular,
$$
1 \ = \ |S_0| \ \le \ |S_1| \ \le \ \cdots \ \le \ |S_{q-1}|.
$$
\end{lemma}
\begin{proof} Straightforward from the definitions.
\end{proof}

%
%
\begin{proposition}\label{S_i+S_j} For all $i,j \ge 1$, we have a \emph{weak grading} as follows:
\begin{eqnarray*}
S_1 + S_j & \subseteq & \hspace{1.5cm} S_{1+j} \ \cup \ S_{1+j+1} \ \ \textrm{ for } j \ge 1, \\
S_i + S_j & \subseteq & S_{i+j-1} \ \cup \  S_{i+j} \ \cup \ S_{i+j+1} \ \ \ \textrm{ for  } i,j \ge 2.
\end{eqnarray*}
\end{proposition}
\begin{proof} For $i,j \ge 1$, we have
$$(im-\rho)+(jm-\rho) \ = \ (i+j)m-2\rho \ > \ (i+j-1)m-\rho.$$
Similarly, we have
$$((i+1)m-\rho-1)+((j+1)m-\rho-1) \ < \ (i+j+2)m-\rho-1.$$
This settles the second inclusion. Assume now $i=1$. Since $\min S_1 = m$ and $m+S_j \,\subseteq\, S_{j+1}$, we have
$$
(S_1+S_j) \cap S_j \ = \ \emptyset.
$$
The first inclusion now follows from the second one.
\end{proof}
When the above weak grading happens to be a true grading up to level $q-1$, more precisely if 
$$
S_i + S_j \ = \ S_{i+j}
$$
for all $i,j \ge 0$ such that $i+j \le q-1$, Wilf's conjecture can be shown to hold in this instance. See Theorem~\ref{graded}.

\bigskip

The following estimate, limiting the size of $(S_i+S_j) \,\cap\, S_{i+j-1}$ by $\rho=\rho(S)$, will play a somewhat subtle role later on.

\begin{proposition}\label{intersections} For all $i,j \ge 1$, we have
\begin{eqnarray*}
|(S_i+S_j) \,\cap\, S_{i+j-1}| & \le & \rho, \\
|(S_i+S_j) \,\cap\, S_{i+j+1}| & \le & m-\rho-1.
\end{eqnarray*}
\end{proposition}
\begin{proof}
We have
$$
S_i+S_j \ \subseteq \ [(i+j)m-2\rho, (i+j+2)m-2\rho-1[.
$$
It follows that
\begin{eqnarray*}
(S_i+S_j) \,\cap\, S_{i+j-1} & \subseteq & [(i+j)m-2\rho, (i+j)m-\rho[\\
(S_i+S_j) \,\cap\, S_{i+j+1} & \subseteq & [(i+j+1)m-\rho, (i+j+2)m-2\rho-1[. \qedhere
\end{eqnarray*}
\end{proof}

\medskip

\subsection{The profile of a numerical semigroup}
It is useful to record how many primitive elements there are in the various levels $S_j$. 

\begin{notation}{ }
For $j \ge 1$, let
$$
\begin{array}{cccccc}
P_j & = & P \cap S_j, &  p_j & = & |P_j|, \\
D_j & = & D \cap S_j, &  d_j & = & |D_j|.
\end{array}
$$
\end{notation}

\medskip

Note that $p_1 \ge 1$ since $m \in P_1$. Note also that $S_1=P_1$, i.e. $D_1 = \emptyset$, as $x \in D$ implies $x \ge 2m$.

\medskip

\begin{definition} The \emph{profile} of $S$ is the $(q-1)$-uple
$$
(p_1,\dots,p_{q-1}) \ \in \ \N^{q-1}.
$$
\end{definition}

\medskip

It may be shown that any $(p_1,\dots,p_{q-1}) \in \N^{q-1}$ with $p_1 \ge 1$ is the profile of a suitable numerical semigroup $S$. For constructing such an $S$, one should start with $m(S) \ge p_1+\dots+p_{q-1}$ at the very least, but the larger the difference $m-\sum p_i$ is, the more room there is for the construction of $S$. For instance, one may start with $P_1=[m,m+p_1[$, $P_2=[2(m+p_1), 2(m+p_1)+p_2[$, and so on.

\subsection{Left and right primitives}
Among the primitive elements of the numerical semigroup $S$, we distinguish the \textit{left ones}, namely those smaller than $c$, and the \textit{right ones}, those contained in $[c,c+m[$. That is, the left primitives are the elements of $P \cap L$, and the right ones are those belonging to $P_q = P \cap I_q$. This covers all of $P$, since $P \subseteq [m,c+m[ \, \subseteq L \cup I_q$.

Note that the right primitives are entirely determined by the left ones together with $c$, in the following sense. In $S_q=I_q$, all decomposable elements are sums of left primitives only. Thus, the right primitives are those elements in $I_q$ which are not attained by sums of left primitives. That is, we have
$$
P_q \ = \ I_q \setminus D.
$$
Or equivalently,
\begin{equation}\label{P cap L}
S \ = \ \vs{P \cap L} \cup [c, \infty[,
\end{equation}
since $P_q = P \cap [c,\infty[$. This specificity of $P_q$ was our reason not to include its cardinality $p_q$ in the profile $(p_1,\dots,p_{q-1})$ of $S$. Incidentally, note that $p_q$ is 
the \emph{down degree} of the vertex $S$ in the tree of all numerical semigroups. (See e.g. \cite{Br08, Br09, RG}.)

\medskip

The description of $S$ by \eqref{P cap L} justifies introducing a specific notation.
\begin{notation} For any nonempty subset $A \subseteq \N^*$ and $c \in \N^*$, we set
$$
\vs{A}_c \ =\ \vs{A} \cup [c,\infty[ \, \ =\ \vs{A \cup [c,c+m[},
$$
where $m = \min A$. It is a numerical semigroup of multiplicity at most $m$ and conductor at most $c$.
\end{notation}

For example, consider the numerical semigroup
$$
S \ = \ \vs{10,15}_{23} \ = \ \vs{10,15} \cup [23, \infty[.
$$
Its left primitives are 10 and 15 and its conductor is $23$. We have $q = \lceil 23/10 \rceil = 3$, and the decomposable elements in $S_3=[23,33[$ are 25 and 30. Therefore, the right primitives in $S$ are 23,24,26,27,28,29,31,32. That is, we have
$$
\vs{10,15}_{23} \ = \ \vs{10,15,23,24,26,27,28,29,31,32}.
$$
Note that the conductor of the semigroup $S=\vs{A}_c$ may occasionally be strictly smaller than $c$. This happens exactly when $S'=\vs{A}$ is itself a numerical semigroup (equivalently, when $\gcd(A) = 1$) whose conductor $c(S')$ is strictly smaller than  $c$. In that case, we simply have $\vs{A}_c \ = \ \vs{A}$. For instance, we have $\vs{3,5}_{10}= \vs{3,5}_{8}=\vs{3,5}$ with conductor 8, and $\vs{3,5}_{7}=\vs{3,5,7}=\vs{3}_5$ with conductor 5.

\subsection{The constant $W_0(S)$}

The number $p_q$ of right primitives is involved in two terms in the formula $W(S) =  |P||L|-c = |P||L|-qm+\rho$. Indeed, we have 
\begin{eqnarray*}
|P| & = & |P \cap L| + p_q, \\
m & = &p_q + d_q,
\end{eqnarray*}
since $m = \big|[c,c+m[\big| = |I_q| = p_q + d_q$. Factoring out $p_q$ from $W(S)$ gives rise to the following closely related constant.

\begin{definition} Let $S$ be a numerical semigroup. We denote
$$
W_0(S) \ = \ |P \cap L| |L|-q d_q+\rho.
$$
\end{definition}
As a side remark, note that $|P \cap L| = p_1+\dots+p_{q-1}$, the sum of the entries of the profile of $S$.  By construction, we have
\begin{equation}\label{W vs W0}
W(S) \ = \ p_q(|L|-q) + W_0(S).
\end{equation}

\begin{proposition}\label{prop W vs W0} Let $S$ be a numerical semigroup. Then $$W(S) \ \ge \ W_0(S).$$ In particular, if $W_0(S) \ge 0$, then $S$ satisfies Wilf's conjecture.
\end{proposition}

\begin{proof}
We have $|L| \ge q$ since \ $L \supseteq \{0,m,\dots,(q-1)m\}$. The stated inequality now follows from \eqref{W vs W0}.
\end{proof}

\medskip

As an application, we will settle Wilf's conjecture for $q=3$ precisely by showing that the stronger inequality $W_0(S) \ge 0$ always holds in this case. 



\begin{remark}\label{W0 ge 0} The inequality $W_0(S) \ge 0$ is equivalent to the fact that $d_q$, the number of decomposables in $I_q = [c,c+m[$, is bounded above as follows:
$$
q d_q \ \le \ |P \cap L||L| + \rho.
$$
\end{remark}

\subsection{$W_0(S)$ may be negative}

While the inequality $W_0(S) \ge 0$ will be shown to hold for $q \le 3$, it no longer holds in general for $q \ge 4$. The first counterexamples were discovered by Jean Fromentin \cite{F}, who showed by exhaustive computer search that all the $33,474,094,027,610$ numerical semigroups $S$ of genus $g \le 60$ do satisfy $W_0(S) \ge 0$ \textit{except in exactly five instances}, namely
$$
\vs{14, 22, 23}_{56}\, , \ \ \vs{16, 25, 26}_{64}\, , \ \ \vs{17, 26, 28}_{68}\, , \ \ \vs{17, 27, 28}_{68}\, \ \mbox{ and } \ \vs{18, 28, 29}_{72}
$$
of genus 43, 51, 55, 55 and 59, respectively. These sole counterexamples up to genus 60  all satisfy $W_0(S)=-1$, $c=4m$ and $W(S) \ge 35$. As a corollary \cite{FH}, it follows that Wilf's \emph{conjecture is true up to genus 60}.

\medskip
The case $W_0(S) < 0$ seems to be very rare indeed. An interesting problem would be to characterize all numerical semigroups $S$ belonging to it.

%
%
%

%
%

%

\subsection{The case $q=2$}
It was shown in \cite{K} that Wilf's conjecture holds for $q = 2$, i.e. in case $m < c \le 2m$. Here is a short proof of a slightly stronger statement.
\begin{proposition} Let $S$ be a numerical semigroup with $q=2$, i.e.  with $c=2m-\rho$ and $\rho \in [0,m-1[$. Then 
$$W_0(S) \ \ge \ \rho \ \ge \ 0.$$
\end{proposition}
\begin{proof}
Let $k = p_1$. Then $|L|=1+k$, since $L = S_0 \mathbin{\dot{\cup}} S_1 = \{0\} \mathbin{\dot{\cup}} P_1$ here. Now
\begin{eqnarray*}
W_0(S)-\rho & = & |P \cap L||L|-2d_2\\
& = & k(1+k) -2 d_2.
\end{eqnarray*}
But
$$
d_2 \ \le \ k(k+1)/2,
$$
since any decomposable element in $S_2=[c,c+m[$ is a sum of two primitives in $P_1$. Therefore $W_0(S) - \rho \ge 0$.
\end{proof}

%

\section{Ap\'ery elements} \label{apery}
Throughout this section again, $S$ denotes a numerical semigroup with multiplicity $m$, conductor $c$ and associated constants $q, \rho$. We shall set up formulas for $|L|$ and $d_q$ involving Ap\'ery elements with respect to $m=m(S)$, in the spirit of those of Selmer~\cite{Se}.

\begin{definition} An \emph{Ap\'ery element} (with respect to $m$) is an element $x \in S$ such that $x-m \notin S$. We shall denote by $X \subset S$ the set of all Ap\'ery elements of $S$.
\end{definition}

Note that a common notation for $X$ is Ap$(S,m)$. It follows from the definition that $X$ is contained in $[0,c+m[$ and contains both extremities $0$ and $c+m-1$. Moreover, we have $|X| = m$. Indeed, for every class $\lambda$ mod $m$, there is a unique $a \in X$ of class $\lambda$, namely the smallest element of that class in $S$. Note also that 
$$P \setminus \{m\} \ \subseteq \ X,$$ since clearly a primitive element cannot belong to $m+S$, except $m$ itself.

\begin{notation} We denote by $N \subset S$ the set of non-Ap\'ery elements, i.e. 
$N = S \setminus X$.
\end{notation}
For example, we have $m \in N$. It is clear that $S+N \subseteq N$. Note also that $N$ and $X$ may equivalently be described as $N = m + S$ and $X = S \setminus N$.

%
%
%

\begin{notation} For all $0 \le j \le q$, we denote
$$
X_j \ = \ X \cap S_j.
$$
\end{notation}
%
For instance, we have
$$
X_0 \ = \ \{0\}, \quad X_1 \ = \ S_1 \setminus \{m\}, \quad X_2 \ \subseteq \ 2X_1 \mathbin{\dot{\cup}} P_2.
$$

\subsection{A formula for $W_0(S)$}
Here is a useful formula for $W_0(S)$ in terms of the cardinalities of the $X_i$'s.

\begin{notation} For $0 \le i \le q$, we denote
$$
\alpha_i \ = \ \left\{
\begin{array}{ll}
|X_i| & \mbox{ if } \ i \le q-1, \\
|X_q \setminus P| & \mbox{ if } \ i = q.
\end{array}
\right.
$$
\end{notation}
In particular, if $q \ge 2$, we have 
\begin{equation}\label{alpha_i ge p_i}
\alpha_0 \ = \ 1, \quad \alpha_1 \ = \ p_1-1, \quad \alpha_i \ \ge \ p_i \ \ \ \mbox{ for all } \ 2 \le i \le q-1,
\end{equation}
since all primitives except $m$ are Ap\'ery elements. But note that $\alpha_q$ only counts the \emph{decomposable} Ap\'ery elements in $S_q$, ignoring $P_q$. Since $|X|=m$ and since $X_q \setminus P$ may be a strict subset of $X_q$, we have
$$
\alpha_0+\alpha_1+\dots+\alpha_q \ \le \ m.
$$
We now identify the left-hand sum with $d_q=|D_q|$.

\begin{proposition}\label{formulas} Let $S$ be a numerical semigroup. We have
\begin{eqnarray}
d_q & = & \sum_{i=0}^{q} \alpha_i, \label{d_q} \label{d_q} \\
|L(S)| & = & \sum_{i=0}^{q-1} (q-i) \alpha_i. \label{aperys in L}
\end{eqnarray}
\end{proposition}
\begin{proof}
On the one hand, we have
$$
m \ = \ |X| \ = \ \sum_{i=0}^{q} |X_i| \ = \ \sum_{i=0}^{q-1} \alpha_i + (\alpha_q+p_q).
$$
On the other hand, we have $m = |S_q| = p_q+d_q$. Comparing both expressions of $m$ yields formula~\eqref{d_q}. Now, by definition of the Ap\'ery elements, for $1 \le i \le q-1$ we have
$$
S_i \ = \ (m+S_{i-1}) \,\mathbin{\dot{\cup}}\, X_i,
$$
and hence
\begin{equation}\label{alpha_i}
|S_i| \ = \ |S_{i-1}|+\alpha_i.
\end{equation}
Since $|L| = |S_0| + |S_1| + \dots + |S_{q-1}|$, it follows by a repeated application of \eqref{alpha_i} that 
$$
|L| \ = \ q + (q-1)\alpha_1+ \dots + \alpha_{q-1},
$$
as desired. \end{proof}
%
%
\begin{corollary} We have
$$
W_0(S)-\rho \ = \ \big(\sum_{i=0}^{q-1} p_i\big)\big(\sum_{i=0}^{q-1} (q-i) \alpha_i\big)-q\sum_{i=0}^{q} \alpha_i.
$$
\end{corollary}
\begin{proof} Straightforward from the formula $W_0(S)-\rho=|P \cap L||L|-qd_q$ and Proposition~\ref{formulas}.
\end{proof}

\section{The Hilbert function of standard graded algebras} \label{hilbert}

We now turn to standard graded algebras, Hilbert functions thereof, Macaulay's theorem, and a condensed version of it which is well-suited to our subsequent applications to Wilf's conjecture. We start by recalling a few basic definitions. In this section, the notation $[x,\infty[$ refers to the usual \textit{real} intervals.

\begin{definition} A \emph{standard graded algebra} is a commutative algebra $R$ over a field $\K$ endowed with a vector space decomposition $R=\oplus_{i\ge 0}\, R_i$ such that $R_0=\K$, $R_iR_j \subseteq R_{i+j}$ for all $i,j \ge 0$, and which is generated as a $\K$-algebra by finitely many elements in $R_1$.
\end{definition}
It follows from the definition that each $R_i$ is a finite-dimensional vector space over $\K$. Moreover, the fact that $R$ is generated by $R_1$ implies that $R_iR_j=R_{i+j}$ for all $i,j \ge 0$.
\begin{definition} Let $R=\oplus_{i\ge 0}\, R_i$ be a standard graded algebra. The \emph{Hilbert function} of $R$ is the map $i \mapsto h_i$ associating to each $i \in \N$ the dimension
$$
h_i \ = \ \dim_{\K} R_i
$$
of $R_i$ as a vector space over $\K$.
\end{definition}
In particular, we have $h_0=1$, and $R$ is generated as a $\K$-algebra by any $h_1$ linearly independent elements of $R_1$.

\subsection{Macaulay's theorem}

Macaulay's theorem rests on the so-called \emph{binomial representations} of integers. Here is some background information about them. 
\begin{proposition} Let $a \ge i \ge 1$ be positive integers. There are unique integers $a_i > a_{i-1} > \cdots > a_1 \ge 0$ such that 
$$
a = \sum_{j=1}^i \binom{a_j}j. 
$$
\end{proposition}
\begin{proof} See e.g. \cite{BH, P}.
\end{proof}
This expression is called the $i$th \emph{binomial representation of $a$}.

\begin{notation} Let $a \ge i \ge 1$ be positive integers. Let
$ \displaystyle
a = \sum_{j=1}^i \binom{a_j}j
$
be its $i$\emph{th} binomial representation. We then denote
$ \displaystyle
a^{\vs{i}} = \sum_{j=1}^i \binom{a_j+1}{j+1}.
$
\end{notation}
%
Note that the right-hand side is a valid $(i+1)$st binomial representation of some positive integer, namely of the integer it sums to.

\bigskip

Here is Macaulay's classical result which constrains the possible Hilbert functions of standard graded algebras \cite{M}.
\begin{theorem}\label{thm macaulay} Let $R = \oplus_{i \ge 0} R_i$ be a standard graded algebra over a field $\Bbb{K}$, with Hilbert function $h_i = \dim_{\Bbb{K}} R_i$ for all $i \ge 0$. Let $i$ be a positive integer. Then
$$
h_{i+1} \ \le \ h_i^{\vs{i}}.
$$
\end{theorem}
The converse also holds in Macaulay's theorem, but we shall not need it here. That is, satisfying these inequalities for all $i \ge 1$ \textit{characterizes} the Hilbert functions of standard graded algebras. See e.g. \cite{BH, MP,P}.

\medskip
For our applications to Wilf's conjecture, we shall derive from Macaulay's theorem a condensed version of it. To this end we first need some facts concerning binomial coefficients.

\subsection{Some binomial inequalities}
Given $i \in \N$ and $x \in \R$, we denote as usual
$$
\binom{x}{i} \ = \ \frac{x(x-1)\dots(x-i+1)}{i!}
$$
if $i \ge 1$, or else 1 if $i=0$. We shall repeatedly use the following well-known fact.
\begin{lemma}\label{bijection} Let $i \ge 1$ be an integer. Then the map
$\displaystyle x \mapsto \binom{x}{i}$ is an increasing continuous bijection (in fact, a homeomorphism) from $[i-1,\infty [$ to $[0,\infty[$. 
\end{lemma}
\begin{proof} By Rolle's theorem, the derivative of the polynomial 
$f=X(X-1)\cdots(X-i+1)$ is of the form $f'=(X-\lambda_1)\cdots(X-\lambda_{i-1})$ where $j-1 < \lambda_j < j$ for all $1 \le j \le i-1$. Therefore $f$ induces an increasing continuous function from $[i-1,\infty [$ onto $[0,\infty[$.
\end{proof}
Consequently, given $i \ge 1$ and any real number $y \ge 0$, there is a unique real number $x \ge i-1$ such that
$$
y = \binom xi.
$$
Moreover, for any real numbers $u,v \ge i-1$, we have
\begin{equation}\label{bijectio}
u < v \ \ \Longleftrightarrow \ \ \binom ui < \binom vi.
\end{equation}
The following result is due to Lov\'asz \cite{L}.

\begin{lemma}\label{lovasz} Let $r \ge 2$ be an integer, and let $u \ge v \ge w$ be real numbers such that $v \ge r-1$ and $w \ge r-2$. Assume
$ \displaystyle
\binom{u}{r} = \binom{v}{r} +\binom{w}{r-1}. 
$
Then
$ \displaystyle
\binom{u}{r-1} \le \binom{v}{r-1} +\binom{w}{r-2}. 
$
\end{lemma}
This appears as an exercise, with proof, in \cite{L}. It is actually stated in a slightly stronger way, where $r-1$ is replaced throughout the conclusion by any integer $k$ such that $1 \le k \le r-1$. But of course, the two versions are equivalent. 
\begin{proof} See \cite{L}. The hint provided by Lov\'asz is to use the following identity:
\[
\binom{u+v+1}{m} = \sum_{k=0}^m \binom{u+k}{k}\binom{v-k}{m-k}. \hfill \qedhere
\]
\end{proof}

\smallskip

Here is a straightforward consequence that we shall need.
\begin{proposition}\label{ineq} Let $r \ge 1$ be an integer, and let $u \ge v \ge w$ be real numbers such that $v \ge r$ and $w \ge r-1$. Assume
$ \displaystyle
\binom{u}{r} = \binom{v}{r} +\binom{w}{r-1}. 
$
Then
$\displaystyle
\binom{u+1}{r+1} \ge \binom{v+1}{r+1} +\binom{w+1}{r}. 
$
\end{proposition}
\begin{proof} We first claim that the following relation holds:
\begin{equation}\label{binom u r+1}
\binom{u}{r+1} \ge \binom{v}{r+1} +\binom{w}{r}. 
\end{equation}
For otherwise, assume on the contrary that the left-hand side were strictly smaller than the right-hand side. Since the function $\displaystyle x \mapsto \binom{x}{r+1}$ is a strictly increasing bijection from $[r,\infty[$ to $[0,\infty[$, there would exist $z > u$ such that
$$
\binom{u}{r+1} < \binom{z}{r+1} = \binom{v}{r+1} +\binom{w}{r}. 
$$
Lemma~\ref{lovasz} would then imply
$$
\binom{z}{r} \le \binom{v}{r} +\binom{w}{r-1},
$$
which is absurd since by hypothesis, the right-hand side equals $\displaystyle \binom{u}{r}$ and $z > u$.
Now, adding $\displaystyle \binom{u}{r}$ to \eqref{binom u r+1}, the hypothesis implies
$$
\binom{u}{r+1}+\binom{u}{r} \ge \binom{v}{r+1} +\binom{w}{r} + \binom{v}{r} +\binom{w}{r-1}
$$
which in turn, by the basic Pascal triangle identity, yields the claimed inequality.
\end{proof}

\subsection{An upper bound on $a^{\vs{i}}$}
We shall also need the following upper bound on $a^{\vs{i}}$. 
\begin{theorem}\label{binom(x,i)} Let $a \ge 0$, $i \ge 1$ be integers, and let $x \ge i-1$ be the unique real number such that
$ \displaystyle
a = \binom{x}{i}.
$
Then
$ \displaystyle
a^{\vs{i}} \le \binom{x+1}{i+1}.
$
\end{theorem}
\begin{proof}
By induction on $i$. For $i=1$, we have $x=a$ and the statement directly follows from the definition. Assume now $i \ge 2$ and the statement true for $i-1$. Consider the $i$th binomial representation of $a$:
$$
a = \sum_{j=1}^i \binom{a_j}j = \binom{a_i}{i} +b,
$$
where
$$
b = \sum_{j=1}^{i-1} \binom{a_j}j.
$$
By definition of the operation $t \mapsto {t}^{\vs{i}}$, we have
$$
a^{\vs{i}} = \binom{a_i+1}{i+1} +b^{\vs{i-1}}.
$$
Let $y \ge i-2$ be the unique real number such that
$\displaystyle
b = \binom{y}{i-1}.
$
Then
\begin{equation}\label{step}
a  \ = \ \binom xi \ = \ \binom{a_i}{i} + \binom{y}{i-1}.
\end{equation}
By the induction hypothesis, we have
$\displaystyle
b^{\vs{i-1}} \le \binom{y+1}{i}.
$
It follows that
$$
a^{\vs{i}} \ \le \ \binom{a_i+1}{i+1} + \binom{y+1}{i}.
$$
But now, it follows from \eqref{step} and Proposition~\ref{ineq} that
$$
\binom{x+1}{i+1} \ \ge \ \binom{a_i+1}{i+1} + \binom{y+1}{i}.
$$
This concludes the proof of the theorem.
\end{proof}

\subsection{A condensed version of Macaulay's theorem}

We now express Macaulay's theorem in a condensed version which is well suited to our present purposes. It is inspired by a similarly condensed version of the Kruskal-Katona theorem, due to Lov\'asz, again given as an exercise in his book \cite{L}. See also the book \cite{Bo} of Bollob\'as, where it is nicely presented and where we first spotted it. 

\begin{theorem}\label{short macaulay} Let $R = \oplus_{i \ge 0} R_i$ be a standard graded algebra over the field $\Bbb{K}$, with Hilbert function $h_i = \dim_{\Bbb{K}} R_i$ for all $i \ge 0$. Let $r \ge 1$ be an integer. Let $x \ge r-1$ be the unique real number satisfying
$\displaystyle
h_r = {x \choose r}.
$
Then
$$
h_{r-1} \ \ge \ {x-1 \choose r-1} \ \ \mbox{ and } \ \ h_{r+1} \ \le \ {x+1 \choose r+1}.
$$
\end{theorem}
\begin{proof} Let $a=h_r$. By Macaulay's Theorem~\ref{thm macaulay} followed by Theorem~\ref{binom(x,i)}, we have 
$\displaystyle
h_{r+1} \ \le \ a^{\vs{r}} \ \le \ \binom{x+1}{r+1}.
$
Assume now, for a contradiction, that 
\begin{equation}\label{ineq 2}
h_{r-1} \ < \ \binom{x-1}{r-1}.
\end{equation}
Let then $y \ge r-2$ be the unique real number such that 
$\displaystyle
h_{r-1}  =  \binom{y}{r-1}.
$
Then $y < x-1$ by Lemma~\ref{bijection}. It would then follow from the statement just proved and Lemma~\ref{bijection} that
$$
h_{r} \ \le \ \binom{y+1}{r}  \ < \ \binom{x}{r},
$$
contrary to our hypothesis. Therefore \eqref{ineq 2} is absurd and we are done.
\end{proof}

\subsection{Averaging the Hilbert function}
 We conclude this section with a result on the average of initial values of the Hilbert function of a standard graded algebra, namely that for any $q \ge 1$, the average of the $h_i$'s for $0 \le i \le q-1$ is bounded below by the ratio $h_q/h_1$. Note the similarity of the formula below with that of Remark~\ref{W0 ge 0}. This will be used in Section~\ref{further} to verify one further case of Wilf's conjecture. 

\begin{theorem}\label{average} Let $R = \oplus_{i \ge 0} R_i$ be a standard graded algebra over the field $\Bbb{K}$, with Hilbert function $h_i = \dim_{\Bbb{K}} R_i$ for all $i \ge 0$. Let $q \ge 1$ be an integer. Then
$$
qh_q \ \le \  h_1\big(1+h_1+\dots+h_{q-1}\big).
$$
\end{theorem}

\begin{proof} Let $x \ge q-1$ be the unique real number such that
$\displaystyle
h_q  =  \binom xq.
$
By repeatedly applying Theorem~\ref{short macaulay} together with Lemma~\ref{bijection}, we get
\begin{equation}\label{h(q-i)}
h_{q-i} \ \ge{} \ \binom{x-i}{q-i}
\end{equation}
for all $0 \le i \le q$. Summing over all $i$ in this range, this implies
$$
\sum_{i=1}^q h_{q-i} \ \ge \ \sum_{i=1}^q \binom{x-i}{q-i}.
$$
Now the sum on the right-hand side is equal to $\displaystyle \binom{x}{q-1}$. Therefore, we have
$$
\sum_{i=1}^q h_{q-i} \ \ge \ \binom{x}{q-1}.
$$
By the identity
$$
\binom{x}{q-1} \ =  \ {} \frac{q}{x-q+1}\binom{x}{q},
$$
it follows that
$$
(x-q+1)\sum_{i=1}^q h_{q-i} \ \ge \ q\binom{x}{q} = q h_q.
$$
And finally, it follows from \eqref{h(q-i)} at $i=q-1$ that $h_1 \ge x-q+1$, yielding the announced inequality.
\end{proof}

\section{Wilf's conjecture for $q=3$} \label{wilf}
We now settle Wilf's conjecture for numerical semigroups satisfying $q = 3$, i.e. $2m < c \le 3m$. The profile of any such semigroup is of the form $(p_1,p_2)$ with $p_1,p_2 \in \N$  and $p_1 \ge 1$. Our first step consists in reducing the verification of the conjecture to the case $p_2=0$. Macaulay's theorem, or its condensed version, will then be needed in the more difficult remaining step, that of settling the case of profile $(p_1,0)$.

\begin{notation} For a subset $A \subseteq \Z$ and an integer $i \ge 1$, we shall denote by $iA$ the $i$th iterated sumset
$$
iA \ = \  \underbrace{A + \dots + A}_i.
$$
\end{notation}
Thus $2P_2=P_2+P_2$ for instance, as involved below. 

\subsection{Reduction to profile $(p_1,0)$}

The announced reduction is relatively straightforward, except that the constant $\rho=\rho(S)$ plays a somewhat subtle role and must be treated with sufficient care. 

%
%

\begin{proposition}\label{reduction} Let $S$ be a numerical semigroup with profile $(p_1,p_2)$. Let $S'=\vs{P_1}_c=\vs{P_1} \cup [c, \infty[$\ , so that $S' \subseteq S$ has profile $(p_1,0)$ and same multiplicity $m$ and conductor $c$ as $S$. Then
$$
W_0(S) \ \ge \ W_0(S')-\rho.
$$
\end{proposition}
\begin{proof} 
Consider the decomposable elements of $S$ in $I_q=I_3$. We have
$$
D_3(S) \ = \ D_3(S') \cup \big((P_1+P_2) \cap I_3\big) \cup \big(2P_2 \cap I_3\big).
$$
Thus, if follows from Proposition~\ref{intersections} involving $\rho$, and the obvious sumset estimates $|2A| \le |A|(|A|+1)/2$ and $|A+B| \le |A||B|$ for finite subsets $A,B \subset \Z$, that
\begin{eqnarray*}
d_3(S) & \le & d_3(S') + |(P_1+P_2)\cap I_3| + |2P_2 \cap I_3| \\
& \le & d_3(S') +p_1p_2 + \min(\rho, p_2(p_2+1)/2).
\end{eqnarray*}
Plugging this inequality in the expression of $W_0(S)$, we get
\begin{eqnarray*}
W_0(S) & = & |P \cap L||L| - 3d_3 + \rho \\
& \ge & |P \cap L||L| - 3d_3(S') -3p_1p_2 -3 \min(\rho, p_2(p_2+1)/2) + \rho.
\end{eqnarray*}

\smallskip
\noindent
\textbf{Claim}. For the sum of the last two terms, the following bound holds:
\begin{equation}\label{inequa}
-3 \min(\rho, p_2(p_2+1)/2) + \rho \ \ge \ - p_2(p_2+1).
\end{equation}
Indeed, if $\rho \le p_2(p_2+1)/2$, then $\min(\rho, p_2(p_2+1)/2) = \rho$, whence
$$
-3 \min(\rho, p_2(p_2+1)/2) + \rho \ = \ -2\rho \ \ge \ - p_2(p_2+1).
$$
Similarly, if $\rho > p_2(p_2+1)/2$, then $\min(\rho, p_2(p_2+1)/2) = p_2(p_2+1)/2$, whence
$$
-3 \min(\rho, p_2(p_2+1)/2) + \rho \ = \ -3p_2(p_2+1)/2 + \rho \ > \ -2p_2(p_2+1)/2.
$$
This establishes the claim.

\bigskip

Plugging \eqref{inequa} into the above estimate of $W_0(S)$, we get
\begin{equation}\label{inequa2}
W_0(S) \ \ge \ |P \cap L||L| - 3d_3(S') -3p_1p_2 - p_2(p_2+1).
\end{equation}
Now, we have $|P \cap L|=p_1+p_2$ and $|L|=1+p_1+(p_2+d_2)$. It follows that
\begin{eqnarray*}
|P \cap L||L| - 3d_3(S') & = & (p_1+p_2)(1+p_1+p_2+d_2) - 3d_3(S') \\
& = & p_2^2 + p_2(1+2p_1+d_2) + p_1(1+p_1+d_2) - 3d_3(S') \\
& = & p_2^2 + p_2(1+2p_1+d_2) + W_0(S')-\rho,
\end{eqnarray*}
by definition of $W_0(S')$ and since $D_2(S) = D_2(S')$. Going back to \eqref{inequa2}, the above yields
\begin{eqnarray*}
W_0(S) & \ge & |P \cap L||L| - 3d_3(S') -3p_1p_2 - p_2(p_2+1) \\
 & = & p_2^2 + p_2(1+2p_1+d_2) + W_0(S')-\rho  -3p_1p_2 - p_2(p_2+1) \\
 & = &  p_2(d_2-p_1) + W_0(S')-\rho.
\end{eqnarray*}
Finally, since $m+P_1 \subseteq D_2$, we have
$
d_2  \ge  p_1.
$
It follows that
$
W_0(S) \ge W_0(S')-\rho,
$
as claimed.
\end{proof}

\medskip

Consequently, in order to settle Wilf's conjecture for the case $q=3$, it remains to prove $W_0(S') \ge \rho$ for any numerical semigroup $S'$ with profile $(k,0)$. This is done in Theorem~\ref{profile (k,0)} below. We start with a counting lemma whose proof relies on our condensed version of Macaulay's theorem.

\subsection{Counting some Ap\'ery elements}

We shall need the following bound relating the numbers of Ap\'ery elements in $2X_1 \cap X_2$ and in $3X_1 \cap X_3$ in a numerical semigroup $S$ of the desired profile.
\begin{lemma}\label{X2X3} Assume the profile of $S$ is $(k,0)$. Let $x \in \R$ be such that $x \ge 1$ and 
$$|2X_1 \cap X_2| \ = \ {x \choose 2}.$$
Then
$$|3X_1 \cap X_3| \ \le \ {x+1 \choose 3}.$$
\end{lemma}
\begin{proof} It suffices to construct a standard graded algebra $R'$ with the property that
$$
\dim R'_i \ = \ |iX_1 \cap X_i|
$$
for $i=1,2$ and then apply Macaulay's theorem or its condensed version. We now proceed to construct such an algebra $R'$.

By hypothesis on the profile of $S$, we have $P\cap L = P_1 = \{m=a_1 <a_2 < \dots < a_k\}=\{m\} \mathbin{\dot{\cup}} X_1$. 
Consider the standard graded algebra
$$
R \ = \ \K[t^{a_1}u, \dots, t^{a_k}u],
$$
where the variables $t$ and $u$ have degree 0 and 1, respectively. Let $A=P_1$. Then, for all $i \ge 0$, we have
$$
\dim R_i \ = \ |i A|.
$$
Now of course, we have
\begin{eqnarray*}
2A & = & (2A \cap X_2) \mathbin{\dot{\cup}} (2A \setminus X_2), \\
3A & = & (3A \cap X_3) \mathbin{\dot{\cup}} (3A \setminus X_3).
\end{eqnarray*}
Moreover, since
$$
2A \ = \ (m+A) \cup 2X_1 \ \ \mbox{ and } \ \ (m+A) \cap X_2 \ = \ \emptyset, 
$$
we have $2A \cap X_2 = 2X_1 \cap X_2$. Similar properties hold for $3A \cap X_3$. Thus, we obtain the following partitions:
\begin{eqnarray*}
2A & = & (2X_1 \cap X_2) \ \mathbin{\dot{\cup}} \ (2A \setminus X_2), \\
3A & = & (3X_1 \cap X_3) \ \mathbin{\dot{\cup}} \ (3A \setminus X_3).
\end{eqnarray*}

Consider the ideal $J \subseteq R$ spanned by all monomials of the form
$$
t^{b}u^2 \ \textrm{ and } \ t^cu^3, 
$$
where
$$
b \in 2A \setminus X_2 \ \ \mbox{ and } \ \ c \in 3A \setminus X_3.
$$
Let 
$$
R' \ = \ R/J.
$$
It is still a standard graded algebra. Regarding its Hilbert function, we claim:
\begin{eqnarray*}
\dim R'_2 & = & |2 X_1 \cap X_2|, \\ 
\dim R'_3 & = & |3 X_1 \cap X_3|.
\end{eqnarray*}
The first equality follows from the above partition $2A = (2X_1 \cap X_2) \mathbin{\dot{\cup}} (2A \setminus X_2)$. The second one follows from the analogous partition $3A = (3X_1 \cap X_3) \mathbin{\dot{\cup}} (3A \setminus X_2)$ and the following inclusion, which shows that killing the monomials $t^bu^2$ of  $J$ in the quotient $R/J$ does not kill any monomial of the form $t^du^3$ for $d \in X_3$: 
\begin{equation}\label{obs}
A+(2A \setminus X_2) \ \subseteq \ 3A \setminus X_3.
\end{equation}
Indeed, we have $2A \setminus X_2 \subseteq (m+S) \cup I_3$, i.e., any $z \in 2A \setminus X_2$ either is not an Ap\'ery element or belongs to $I_3$. Inclusion \eqref{obs} now follows from the inclusions
\begin{eqnarray*}
A + (m+S) & \subseteq & m+S, \\
A + I_3 & \subseteq & I_\infty,
\end{eqnarray*}
where $I_\infty = \bigcup_{j \ge 4} I_j=[c+m, \infty[$, and the fact that $X_3$ is disjoint from both $m+S$ and $I_\infty$.

\medskip

The lemma now follows by applying the condensed Macaulay Theorem~\ref{short macaulay} to the claimed respective dimensions of $R'_2, R'_3$.
\end{proof}

\subsection{The case of profile $(k,0)$}

\begin{theorem}\label{profile (k,0)} Let $S \subset \N$ be a numerical semigroup with $q=3$ and profile $(k,0)$ for some $k \ge 1$. Then $W_0(S) \ge \rho(S)$.
\end{theorem}
\begin{proof} By hypothesis, we have $P\cap L = P_1 = \{m\} \mathbin{\dot{\cup}} X_1$. Let us denote
$$X_1 \ = \ \{a_2 < \dots < a_k\}$$
with $m < a_2$. We may list the elements of $D_3$ in terms of the Ap\'ery ones as follows:
$$
D_3 \ = \ \{3m\} \mathbin{\dot{\cup}} \big(2m+X_1) \mathbin{\dot{\cup}} \big(m+X_2) \mathbin{\dot{\cup}} X_3',
$$
where $X_3'=X_3 \setminus P$. By Proposition~\ref{formulas}, and recalling our notation $\alpha_2=|X_2|$, $\alpha_3=|X_3'|$, we have
\begin{eqnarray*}
d_3 & = & k + \alpha_2+\alpha_3, \\
|L| & = & 3+2(k-1)+\alpha_2 \\
& = & 2k+1+\alpha_2.
\end{eqnarray*}
Therefore
\begin{eqnarray*}
W_0(S)-\rho & = & k|L|-3d_3 \\
& = & k\big(2k+1+\alpha_2\big)-3(k+\alpha_2+\alpha_3) \\
& = & 2k(k-1)+k\alpha_2-3(\alpha_2+\alpha_3) \\
& = & 4\binom{k}{2}+k\alpha_2-3(\alpha_2+\alpha_3).
\end{eqnarray*}
We now proceed to bound $\alpha_2+\alpha_3 = |X_2|+|X_3'|$. Since $X_2 \subseteq 2X_1$ and $X_3' \subseteq 2X_1 \cup 3X_1$, we have
\begin{eqnarray*}
\alpha_2 & = & |X_2| \ = \ |2X_1 \cap X_2|, \\
\alpha_3 & = & |X_3'| \ = \ |2X_1 \cap X_3|+|3X_1 \cap X_3|. 
\end{eqnarray*}
It follows that
\begin{eqnarray*}
\alpha_2+\alpha_3 & = & |2X_1 \cap X_2|+|2X_1 \cap X_3|+|3X_1 \cap X_3| \\
& \le & |2X_1| + |3X_1 \cap X_3| \\
& \le & \binom{k}{2} + |3X_1 \cap X_3|.
\end{eqnarray*}
Plugging this into the latter estimate of $W_0(S)-\rho$, we get
\begin{equation}\label{estimate q=3}
W_0(S)-\rho \ \ge \ \binom{k}{2}+k|2X_1 \cap X_2|-3|3X_1 \cap X_3|.
\end{equation}
Let $x \ge 1$ be the unique real number such that 
$$
|2X_1 \cap X_2| \ = \ \binom{x}{2}.
$$
Note that $x \le k$, since
$$
|2X_1 \cap X_2| \ \le \ |2X_1| \ \le \ \binom{k}{2}.
$$
Further, it follows from Lemma~\ref{X2X3} that 
$$
|3X_1 \cap X_3| \ \le \ \binom{x+1}{3}.
$$
Plugging these inequalities into \eqref{estimate q=3}, we obtain
\begin{eqnarray*}
W_0(S) - \rho & \ge & \binom{k}{2}+k\binom{x}{2}-3\binom{x+1}{3} \\
& = & \binom{k}{2}+k\binom{x}{2}-3\frac{x+1}{3}\binom{x}{2} \\
& = & \binom{k}{2}+(k-x-1)\binom{x}{2}.
\end{eqnarray*}
Since $\displaystyle \binom{k}{2} \ge \binom{x}{2}$ and $k \ge x$ as observed above, we conclude
$$W_0(S)-\rho \ \ge \ (k-x)\binom{x}{2} \ \ge \ 0,$$ as desired.
\end{proof}

\begin{corollary}\label{case q=3} Wilf's conjecture holds for all numerical semigroups $S$ satisfying $q(S)=3$.
\end{corollary}
\begin{proof}
Straightforward from the above result and the reduction to profile $(k,0)$ provided by Proposition~\ref{reduction}, which together imply $W_0(S) \ge 0$.
\end{proof}

As observed in the Introduction, the importance of this corollary stems from a recent result of Zhai \cite{Z} stating that, as $g$ goes to infinity, the proportion of numerical semigroups of genus $g$ satisfying $q=3$ tends to 1. As a matter of illustration, here is a table showing how $q$ is distributed for $18 \le g \le 25$. It clearly shows that, in this range for $g$, the two cases $q=3$ and $q=2$ together contain an overwhelming majority of numerical semigroups. This table was obtained with the GAP package {\tt numericalsgps} \cite{DG}.

\renewcommand{\tabcolsep}{1.2pt}
\begin{table}
\caption{Distribution of $q=q(S)$ by genus $g$, for $18 \le g \le 25$ and $q \le 20$.}
\vspace{-0.2cm}
\begin{center}
{\small
\begin{tabular}{|l||c|r|r|r|r|r|r|r|r|r|r|r|r|r|r|r|r|r|r|r|r|}
\hline
$g \backslash q$ & 1 & 2 & 3 & 4 & 5 & 6 & 7 & 8 & 9 & 10 & 11 & 12 & 13 & 14 & 15 & 16 & 17 & 18 & 19 & 20 \\  
\hline
\hline
18 &   1 &  4180 &  6935 &  1739 &  409 &  132 &  37 &  13 &  14 &  2 &  2 &  2 &  0 &  0 &  0 &  0 &  0 &  1  &   &    \\
\hline
19 &  1 &  6764 &  11828 &  2895 &  670 &  195 &  63 &  20 &  14 &  8 &  2 &  2 &  1 &  0 &  0 &  0 &  0 &  0 &  1  & \\
\hline
20 & 1 & 10945 & 20096 & 4805 & 1085 & 290 & 103 & 35 & 14 & 15 & 2 & 2 & 2 & 0 & 0 & 0 & 0 & 0 & 0 & 1 \\
\hline
21 &  1 &  17710 &  34069 &  7943 &  1750 &  453 &  172 &  46 &  19 &  15 &  9 &  2 &  2 &  2 &  0 &  0 &  0 &  0 &  0 &  0 \\
\hline
22 &  1 &  28656 &  57566 &  13108 &  2806 &  707 &  249 &  81 &  32 &  16 &  16 &  2 &  2 &  2 &  1 &  0 &  0 &  0 &  0 &  0 \\
\hline
23 &  1 &  46367 &  96949 &  21509 &  4453 &  1102 &  357 &  132 &  44 &  16 &  17 &  9 &  2 &  2 &  2 &  0 &  0 &  0 &  0 &  0  \\
\hline
24 &  1 &  75024 &  162911 &  35248 &  7052 &  1741 &  500 &  221 &  60 &  26 &  17 &  18 &  2 &  2 &  2 &  2 &  0 &  0 &  0 &  0  \\
\hline
25 &  1 &  121392 &  273139 &  57649 &  11149 &  2648 &  750 &  301 &  100 &  42 &  17 &  18 &  10 &  2 &  2 &  2 &  1 &  0 &  0 &  0  \\
\hline
\end{tabular}
}
\end{center}
\label{default}
\end{table}


\begin{remark} As observed by A. Sammartano after reading a preliminary version of this paper, one can show that the equality case $W(S)=0$ in Wilf's conjecture cannot occur for $q=3$ besides the known ones cited in the Introduction \cite{Sa2}. Indeed, since $W(S)=p_3(|L|-3)+W_0(S)$ and since $W_0(S) \ge 0$ holds for $q=3$, it follows from $W(S)=0$ that  $p_3(|L|-3)=W_0(S)=0$. Moreover, going through the chains of inequalities in the proofs of Proposition~\ref{reduction} and Theorem~\ref{profile (k,0)}, ones sees that the equality $W_0(S)=0$ can only occur if $\rho=p_2(p_2+1)/2$, $m+P_1=D_2$, $|P_1+P_2|=p_1p_2$, $|2P_2|=p_2(p_2+1)/2$, $|2X_1 \cap X_2|=\binom{p_1}{2}$ and $|3X_1 \cap X_3|=\binom{p_1+1}{3}$. Considering all these constraints together, one can show that the profile of $S$ either equals $(1,0)$, or $(1,1)$ provided $p_3=0$, both known equality cases in Wilf's conjecture.
\end{remark}

%
%
%

\section{Further results} \label{further}
Using the present methods, we settle Wilf's conjecture in a few other cases, namely for numerical semigroups $S$ satisfying $S_i+S_j = S_{i+j}$ whenever $i+j \le q-1$, for those satisfying $|L(S)| \le 6$, and finally for those satisfying $\gcd(L(S)) \ge 2$.

\subsection{The case of true grading}

\begin{theorem}\label{graded}
Let $S$ be a numerical semigroup satisfying $S_i + S_j = S_{i+j}$ for all $i+j \le q-1$. Then $W_0(S) \ge \rho \ge 0$, and hence $S$ satisfies Wilf's conjecture.
\end{theorem}
\begin{proof} It follows from the hypothesis that $S_i = iS_1$ for all $1 \le i \le q-1$. Therefore $P \cap L = P_1=S_1$ and $D_q \subseteq qS_1$. Now, denote $S_1 = \{a_1,a_2, \dots, a_k\}$ with $m=a_1 <a_2 < \dots < a_k$. As in the proof of Lemma~\ref{X2X3}, consider the standard graded algebra
$$
R \ = \ \K[t^{a_1}u, \dots, t^{a_k}u],
$$
where the variables $t$ and $u$ have degree 0 and 1, respectively. As Hilbert function of $R$, we have
$$
h_i \ = \ \dim R_i \ = \ |i S_1| \ = \ |S_i|
$$
for all $0 \le i \le q-1$, and $h_q = \dim R_q = |qS_1|$. It follows from Theorem~\ref{average} that
\begin{equation}\label{qh1}
qh_q \ \le \ h_1(1+h_1+\dots+h_{q-1}). 
\end{equation}
Since $W_0(S)-\rho=|P \cap L||L| -qd_q$, since $d_q = |D_q| \le |qS_1| = h_q$, and by the formula for $|L|$ in Lemma~\ref{formula for |L|}, we have
\begin{eqnarray*}
W_0(S) - \rho & \ge & |P \cap L||L| -qh_q \\
& = & h_1(1+h_1+\dots+h_{q-1}) -qh_q.
\end{eqnarray*}
Hence $W_0(S) - \rho \ge 0$ by \eqref{qh1}, as claimed. 
\end{proof}
\begin{corollary} Let $S$ be a numerical semigroup satisfying $q \ge 4$ and 
$$
P \cap L \ \subseteq \ \left[m, m+\frac{m-\rho}{q-1}\right[.
$$
Then $S$ satisfies Wilf's conjecture.
\end{corollary}
\begin{proof} It suffices to show that $S$ satisfies the hypotheses of Theorem~\ref{graded}. First note that 
$$
\left[m, m+\frac{m-\rho}{q-1}\right[ \ \subseteq \ I_1.
$$
Indeed, we have $m+(m-\rho)/(q-1) \le 2m-\rho = \max I_1-1$, since 
\begin{eqnarray*}
(q-1)m + (m-\rho) & \le & (q-1)m + (q-1)(m-\rho) \\
& \le & (q-1)(2m-\rho).
\end{eqnarray*}
It follows that $P \cap L = P_1$. Therefore, for all $2 \le k \le q-1$, we have $S_k = kS_1\cap I_k$.

Consider now the following inclusions for $k$ in this same range:
\begin{eqnarray*}
kS_1 & \subseteq & [km, km + k(m-\rho)/(q-1)[ \\
& \subseteq & [km, km+(m-\rho)[ \\
& \subseteq  & I_k.
\end{eqnarray*}
It follows that $S_k = kS_1$. Therefore, for any integers $1 \le i,j \le q-1$ such that $i+j \le q-1$, we have
$$
S_i+S_j \ = \  iS_1+jS_1 \ = \ (i+j)S_1 \ = \ S_{i+j},
$$
and we are done.
\end{proof}

\begin{example} Let $S$ be a numerical semigroup with $m = 1000$ and $c = 4000$. Assume further that all left primitives of $S$ are contained in $[1000, 1333[$. Equivalently, let $A \subseteq [0,333[$ be an arbitrary subset, and let
$$
S \ = \ \vs{1000+A}_{4000}  \ = \ \vs{1000+A} \cup [4000, \infty[.
$$
Then $S$ satisfies Wilf's conjecture. 

\smallskip
Indeed, we have $q=4$, $\rho = 0$, and $P \cap L \subseteq [1000, 1000+333[$ by hypothesis. Hence the above corollary applies. 

\end{example}

\subsection{The case $|L| \le 6$}

Dobbs and Matthews \cite{DM} settled Wilf's conjecture for numerical semigroups $S$ satisfying $|L| \le 4$. As briefly commented below, that result easily follows from the now settled case $q \le 3$ of the conjecture. We now informally establish Wilf's conjecture in case $|L| \le 6$, and shall extend that result to the case $|L| \le 10$ in a forthcoming publication. 

\begin{proposition}\label{|L| le 6} Numerical semigroups $S$ with $|L(S)| \le 6$ satisfy Wilf's conjecture.
\end{proposition}
\begin{proof}
By Corollary~\ref{case q=3}, it suffices to consider the case $q \ge 4$. So, from now on, we assume $|L| \le 6$ and $q \ge 4$. Let $(p_1,\dots,p_{q-1})$ be the profile of $S$. It follows from Proposition~\ref{formulas} and \eqref{alpha_i ge p_i} that
\begin{equation}\label{L ge}
|L| \ \ge \ 1 +(q-1)p_1+(q-2)p_2+\dots+p_{q-1}.
\end{equation}
In particular, since $|L| \le 6$, and since $p_1 \ge 1$ always, we must have $q \le 6$. Moreover, we must have $p_1=1$, for if $p_1 \ge 2$ then $|L| \ge 7$. Similarly, we must have $p_2 \le 1$, for otherwise $|L| \ge 8$. Therefore, by \eqref{L ge}, the only profiles with $4 \le q \le 6$ and compatible with $|L| \le 6$ are
$$
(1,1,0), \ (1,0,k), \ (1,0,0,k), \ (1,0,0,0,k)
$$
for some small integer $k \ge 0$. We first treat the last three possibilities in one single case.

\smallskip
\noindent
$\bullet$ Assume $S$ is of profile $(1,0,\dots,0,k) \in \N^{q-1}$ with $q \ge 4$ and $k \in \N$. We then claim 
$$
W_0(S) = k(k+1) + \rho,
$$
and so $S$ satisfies Wilf's conjecture. Indeed, one has 
$$
(\alpha_0, \alpha_1, \dots,\alpha_{q-1})=(1,0,\dots,0,k),
$$
as easily seen. We have $|P \cap L|=1+k$, and Proposition~\ref{formulas} yields 
$$
|L| \ = \ q+k, \quad d_q \ = \ 1+k.
$$
Therefore $W_0(S)-\rho = (1+k)(q+k)-q(1+k)=k(1+k)$, and we are done.

\smallskip
\noindent
$\bullet$ Assume now $S$ is of profile $(1,1,0)$, a slightly more delicate case. Here $q=4$, $|P \cap L|=2$, and we have
$$
\alpha_0 \ = \ 1, \quad \alpha_1 \ = \ 0, \quad \alpha_2 \ = \ 1, \quad \alpha_3 \ \le \ 1, \quad \alpha_4 \ \le \ 1,
$$
as easily seen. Thus, by Proposition~\ref{formulas}, we have
$$
|L| \ = \ 6 + \alpha_3, \quad d_4 \ = \ 2 + \alpha_3+\alpha_4.
$$
Therefore $W_0(S)-\rho = 2(6+\alpha_3)-4(2+\alpha_3+\alpha_4) = 4-2\alpha_2-4\alpha_4$. If either $\alpha_3=0$ or $\alpha_4=0$, then $W_0(S)-\rho \ge 0$ and we are done. However, if $\alpha_3=\alpha_4=1$, then $W_0(S)-\rho =-2.$ But in this case, we must have $X_3=2X_2$ and $X_4 \setminus P = 3X_2$. Proposition~\ref{intersections} then implies $\rho \ge 2$, whence $W_0(S) \ge 0$, and we are done again.

\smallskip
This settles, albeit informally, Wilf's conjecture for $|L| \le 6$. 
\end{proof}


As mentioned above, we shall extend the verification of Wilf's conjecture to the case $|L| \le 10$ in a forthcoming publication. More precisely, we shall prove the following result.

\begin{theorem}\label{L le 10} Let $S$ be a numerical semigroup with $|L(S)| \le 10$. Then $W_0(S) \ge \rho$, except possibly if $S$ is of profile $(1,0,1,0)$. In that special profile, we have $W_0(S) \ge \rho-1$, and if equality holds, then $\rho \ge 2$. In any case, $S$ satisfies Wilf's conjecture. 
\end{theorem}

An example where $|L(S)| \le 10$ and $W_0(S)=\rho-1$ is given by $S=\vs{5,13}_{22}$, for which $|L|=7$ and $\rho=3$. Its profile is $(1,0,1,0)$, as expected.

\medskip

The proof of Theorem~\ref{L le 10}, like that of Proposition~\ref{|L| le 6}, combines some general reductions, in the spirit of Proposition~\ref{reduction}, and some ad-hoc arguments for a few specific profiles.

\subsection{The case $\gcd(L(S)) \ge 2$}
Sammartano proved in \cite{Sa} that \textit{if the numerical semigroup $S$ satisfies $e \ge m/2$, then it satisfies Wilf's conjecture}. Here is a straightforward consequence. 

\begin{proposition}\label{A_c} Let $S$ be a numerical semigroup such that $\gcd(L(S)) \ge 2$, i.e. such that the left primitives of $S$ have a nontrivial common factor. Then $S$ satisfies Wilf's conjecture.
\end{proposition}
\begin{proof} Let $k = \gcd(L(S)) = \gcd(P \cap L)$, and assume $k \ge 2$.
Then $D_q$, the set of right decomposable elements in $S_q=I_q$, is entirely contained in $k\N$. Thus $|D_q| \le m/k$. Since 
$$P_q = S_q \setminus D_q$$
and since $|S_q|=m$, it follows that $e \ge |P_q| \ge m-m/k \ge m/2$. The conclusion now follows from Sammartano's result mentioned above.
\end{proof}

As an application, it follows that all \emph{inductive} numerical semigroups satisfy Wilf's conjecture. These are obtained from $S_0 = \N$ by applying finitely many steps of the form $S \mapsto a \cdot S \cup (ab+ \N)$, where $a,b$ are varying positive integers and $a \cdot S = \{as \mid s \in S\}$.

\medskip

The numerical semigroups $S$ satisfying $\gcd(L(S)) \ge 2$ have an interesting geometric interpretation. Let $\mathcal{T}$ denote the tree of all numerical semigroups. Then \textit{a numerical semigroup $S$ satisfies $\gcd(L(S)) \ge 2$ if and only if the subtree $\mathcal{T}_S \subseteq \mathcal{T}$ rooted at $S$ is infinite}. 

\smallskip

Here are some explanations; see also \cite[Theorem 10 in Section 3]{BB}. Recall first that the root of $\mathcal{T}$ is $\N=\vs{1}$, that the father in $\mathcal{T}$ of the numerical semigroup $S \not= \N$ is the numerical semigroup $\widehat{S}=S \mathbin{\dot{\cup}} \{F(S)\}$, and that for all $g \in \N$, the vertices at level $g$ in $\mathcal{T}$ are all numerical semigroups of genus $g$. As mentioned earlier, the down degree of $S$ in $\mathcal{T}$ is the number $p_q$ of right primitives in $S$. For instance, $S$ is a leaf in $\mathcal{T}_S$ if and only if $p_q=0$. Finally, let us denote by $\mathcal{T}_S$ the subtree of $\mathcal{T}$ rooted at $S$. For instance, we have $\mathcal{T}_S = \{S\}$ if and only if $S$ is a leaf in $\mathcal{T}$.

\smallskip
Let us now prove the above characterization. Let $A = L(S)$ and $k=\gcd(A)$. Note first that if $T$ is any descendant of $S$, then $A \subseteq T \subseteq S$ by construction.

$\bullet$ If $k \ge 2$, then $S$ has infinitely many descendants $S'$ in $\mathcal{T}$, e.g. all $S' = \vs{A}_d$ with $d > \max(A)+2$. This is indeed an infinite collection, since if $d_1 < d_2$, the equality $\vs{A}_{d_1} = \vs{A}_{d_2}$ can only occur if $d_1 \equiv 0 \bmod k$ and $d_2=d_1+1$.

$\bullet$ Conversely, if $k = 1$, let $S_0 = \vs{A}$. Then $S_0$ is a numerical subsemigroup of $S$, and any descendant $T$ of $S$ satisfies $S_0 \subseteq T \subseteq S$. Therefore $\mathcal{T}_S$ is finite in this case, as desired.

\bigskip
\noindent
\textbf{Acknowledgment.} It is a pleasure to thank several people for helpful discussions and/or bibliographical references related to this work, including Maria Bras-Amor\'os, Manuel Delgado, Ralph Fr\"{o}berg, Ornella Greco, Jorge Ram\'{\i}rez Alfons\'{\i}n, Giuseppe Valla and Santiago Zarzuela. Special thanks are due to Pedro A. Garc\'{\i}a-S\'anchez and Alessio~Sammartano for com\nolinebreak[4]ments  on an earlier version of this paper and, last but not least, to Jean Fromentin for several lively discussions and computer experiments.


\medskip
\noindent
\textbf{Author's address:}

\small

\noindent
$\bullet$ Shalom Eliahou\textsuperscript{a,b},

\noindent
\textsuperscript{a}Univ. Littoral C\^ote d'Opale, EA 2597 - LMPA - Laboratoire de Math\'ematiques Pures et Appliqu\'ees Joseph Liouville, F-62228 Calais, France\\
\textsuperscript{b}CNRS, FR 2956, France\\
\textbf{e-mail:} eliahou@lmpa.univ-littoral.fr

\end{document}